\newif\ifverbose
\verbosetrue

\newif\ifincludefigure
\includefiguretrue

\documentclass[a4paper]{article}
\usepackage[top=30truemm,bottom=30truemm,left=25truemm,right=25truemm]{geometry}
\usepackage{color}
\usepackage{amssymb} 
\usepackage{amsmath}
\usepackage{mathrsfs} 
\usepackage{graphicx}
\usepackage{amsthm}
\newtheorem{theorem}{Theorem}
\newtheorem{proposition}{Proposition}
\newtheorem{lemma}{Lemma}
\theoremstyle{definition}
\newtheorem{definition}{Definition}
\newtheorem{example}{Example}
\newtheorem{remark}{Remark}

\def\R{\mathbf R} 
\def\C{\mathbf C} 
\def\H{\mathbf H} 
\def\E#1{\mathbf E\left(#1\right)} 
\def\P#1{\mathbf P\left(#1\right)} 
\def\I{ i }

\def\({\left(}
\def\){\right)}
\def\hf{\mathfrak B} 

\title{
An integral formula for the powered sum of
the independent, identically and normally distributed random variables
}
\author{
Tamio Koyama
}
\date{}

\begin{document}
\maketitle
\begin{abstract}
The distribution of the sum of $r$-th power of standard normal random variables
is a generalization of the chi-squared distribution.
In this paper,
we represent the probability density function of the random variable
by an one-dimensional absolutely convergent integral
with the characteristic function.
Our integral formula is expected to be applied
for evaluation of the density function.
Our integral formula is based on the inversion formula,
and we utilize a summation method.
We also discuss on our formula in the view point of hyperfunctions.

\end{abstract}

\section{Introduction}\label{intro}
Let $\{X_k\}_{k\in\mathbf N}$ be a sequence of 
the independent, identically distributed random variables,
and suppose the distribution of each variable $X_k$ is 
the standard normal distribution.
Fix a positive integer $r$ and $n$.
In this paper, we discuss 
the sum of the $r$-th power of standard normal random variables.
This distribution function can be written in the form of 
$n$ dimensional integral:
\begin{equation}\label{a8Jun2018}
F(c)
:=\P{\sum_{k=1}^n X_k^r<c}
=\int_{\sum_{k=1}^n x_k^r<c} \frac{1}{(2\pi)^{n/2}}
\exp\(-\frac{1}{2}\sum_{k=1}^n x_k^2\) dx_1\dots dx_n
\end{equation}
In the case where $r=2$,
this integral equals to the cumulative distribution function of
the chi-squared statistics.
The sum of the $r$-th power $\sum_{k=1}^nX_k^r$ is
an generalization of the chi-squared statistics.
The weighted sum $\sum_{k=1}^n w_kX_k^2$,
where $w_k$ is a positive real number, 
of the squared normal variables
is another generalization of the chi-squared statistics.
Numerical analysis of the cumulative distribution function of
the weighted sum 
is discussed
in \cite{castano-lopez} and \cite{takemura-koyama}.
In \cite{marumo-oaku-takemura2015},
Marumo, Oaku, and Takemura discussed 
the numerical analysis of integral \eqref{a8Jun2018}
in the case where $r=3$.
Their approach was
the holonomic gradient method (\cite{NNNOSTT2011} ,Chapter 6 of \cite{dojo}).
They derived a system of differential equations for
the probability density function of $\sum_{k=1}^nX_k^3$.
By numerically solving the system of differential equations,
they evaluate the probability density function.
As a generalization of
the chi-squared statistics or
the sum of cubes of standard normal random variables,
the sum of the $r$-th power ($r\geq 4$)
of standard normal random variables
is a basic quantity in statistics.
Detailed investigation on the quantity can be
expected to be applied to Hypothesis test,
and we are interested in numerical analysis of
the probability density function of the quantity.

In order to derive the system of differential equations
for the probability density function $f(x)$ of $\sum_{k=1}^nX_k^3$.
Marumo, Oaku, and Takemura utilize
the following inversion formula: 
\begin{equation}\label{formal}
f(x)
=
\frac{1}{2\pi}\int_{-\infty}^\infty e^{-itx}\varphi(t)^ndt,
\end{equation}
where $\varphi(t):=\E{\exp\(\I tX_k^3\)}\,(i:=\sqrt{-1})$
is the characteristic function of $X_k^3$.
Note that Equation \eqref{formal} is only a formal equation
in a naive sence.
Actuary, it is not certain whether
the integral in the right-hand side of \eqref{formal} is convergent or not.
It is hard to prove the convergence of the integral
by utilizing with the Laplace approximation.
The discussion in \cite{marumo-oaku-takemura2015}
does not show the convergence of \eqref{formal}
and justify Equation \eqref{formal} by
the theory of Schwartz distributions.
In this paper, we consider 
a justification of the formal equation \eqref{formal} for general $r\geq 3$
from the viewpoint of summation methods and hyperfunctions.
As a result of the consideration, we derive a formula
which represents the probability density function of $\sum_{k=1}^n X_k^r$
as an one-dimensional absolutely convergent integral.
This formula is expected to be applied for evaluation of the density function.

Before starting regorous discussion,
let us give an intuitive discussion on the formal equation \eqref{formal}.
Here, we consider the case where $\varphi(t)$ is
the characteristic function of a general random varialbe $X$.
Note that the distribution of $X$ can be
the exponential distribution, the binomial distribution
or other distribution.
Dividing the domain of the formal integral
in the right-hand side of \eqref{formal},
it can be written as 
\begin{equation}\label{19}
\frac{1}{2\pi}\int_{-\infty}^0 e^{-itx}\varphi(t)^ndt
+
\frac{1}{2\pi}\int_0^\infty e^{-itx}\varphi(t)^ndt.
\end{equation}
Note that the integral in the first term converges
when $\Im x > 0$ holds,
the second term converges when $\Im x < 0$.
Hence, the first and the second term define holomorphic functions
on the upper half plane and the lower half plane respectively.
Since these holomorphic functions has analytic continuation,
the expression \eqref{19} make a sense for $x\in\R\backslash\{0\}$
if $x$ is not a singular point of the holomorphic functions.
This intuitive explanation can be justified
by a summation methods in the case where
random variable $X$ is the sum of $r$-th power of
standard normal random variables.
For general random variables,
this intuitive explanation is justified
by theory of Fourier hyperfunctions.
Theory of hyperfunction introduced in \cite{sato1958},\cite{kawai1970}
has many applications to linear partial differential equations
(see, e.g., \cite{graf}).
Applications of hyperfunctions to numerical analysis are
discussed in \cite{imai} and \cite{ogata-hirayama}.
Our discussion in this paper is  an application of the theory of hyperfunctions
to statistics. 

The organization of this paper is as follows.
In Section \ref{sec:InvForm},
we review the inversion formulae
in the probability theory and the Fourier analysis,
and we justify the intuitive discussion in Section \ref{intro}
by a summation method.
In Section \ref{sec:hyperfunction},
we review the theory of hyperfunctions of
a single variable based on \cite{sato1958} and \cite{kaneko1988},
and we justify the intuitive discussion in Section \ref{intro}
from the view point of hyperfunctions.
In Section \ref{PowerSum},
we apply the inversion formula given in Section \ref{sec:InvForm}
to the sum of $r$-th power of standard normal random variables.
We also calculate the limit in the boundary-value representation.

\section{Inversion Formula}\label{sec:InvForm}
In this section, we review the inversion formulae
in the probability theory and the Fourier analysis.

In the probability theory,
the following L\'evy's inversion formula is most famous.
\begin{theorem}[L\'evy's inversion formula]
Let $X$ be a random variable
and $\varphi(t)$ be the characteristic function of $X$.
Then, for $a<b$,
$$
\lim_{R\uparrow\infty}\frac{1}{2\pi}
\int_{-R}^R \frac{e^{-ita}-e^{-itb}}{it}\varphi(t)dt
=\frac{1}{2}\P{X=a}+\P{a<X<b}+\frac{1}{2}\P{X=b}
$$
holds.
In the case where $\int_{-\infty}^\infty|\varphi(t)|dt<\infty$ holds,
$X$ has continuous probability density function $f$
and
\begin{equation}\label{c7Jun2018}
f(x)
=
\frac{1}{2\pi}\int_{-\infty}^\infty e^{-itx}\varphi(t)dt
\end{equation}
holds.
\end{theorem}
\begin{proof}
see, e.g., \cite[p175]{williams1991probability}.
\end{proof}

If $X$ has a probability density function $f$, 
then
the characteristic function of $X$ is
the Fourier transformation of $f$, i.e.,
$\varphi(t)=\int_{-\infty}^\infty e^{itx}f(x)dx$.
If $f$ is a rapidly decreasing function,
the inversion formula \eqref{c7Jun2018} is holds immediately.
If $f$ is a square-integrable function,
$\int_{-\infty}^\infty f(x)^2dx < \infty$, 
the inversion formula is 
\begin{equation}\label{b7Jun2018}
f(x)=\frac{1}{2\pi}\lim_{R\rightarrow\infty}\int_{-R}^R e^{-itx}\varphi(t)dt.
\end{equation}
Note that the right hand-side of \eqref{b7Jun2018} is improper integral.
Even if $f$ is square-integrable,
integrand $e^{-itx}\varphi(t)$ may not be Lebesgue integrable,
i.e., 
$\int_{-\infty}^\infty |e^{-itx}\varphi(t)|dt$ can be infinite.
We also note that the probability density function $f$
may not be square-integrable.
For example, the following probability density function $g(x)$
is not square-integrable:
$$
g(x):=
\begin{cases}
\frac{1}{2\sqrt{x}} & (0<x<1)\\
0 & (\text{else})
\end{cases}
$$
If $f$ is a slowly increasing function,
Equation \eqref{c7Jun2018} can be justified
as an equation of Schwartz distributions. 

The following proposition is a justification of
the intuitive discussion in Section \ref{intro}
by a summation method.
\begin{proposition}\label{a7Jun2018}
Let $f$ be a probability density function
and $\varphi(t):=\int_{-\infty}^\infty e^{ity}f(y)dy$ be
the characteristic function of $f$.
For any continuous point $x$ of $f$, the equation
$$
f(x)
=
\frac{1}{2\pi}\lim_{\varepsilon\rightarrow +0}\(
 \int_{-\infty}^0 e^{-i(x+i\varepsilon)t}\varphi(t)dt
+\int_0^\infty   e^{-i(x-i\varepsilon)t}\varphi(t)dt
\)
$$
holds.
\end{proposition}

We utilize a fact in \cite[p28, Note 1.3]{kaneko1988}
with a small change.
\begin{lemma}\label{b8Jun2018}
Let $K(x)$ and $\varphi(x)$ be real-valued functions on $\mathbf R$.
Suppose that $K(x)$ is absolutely integrable and
$\int_{-\infty}^\infty K(x)dx =1$ holds,
and that $\varphi(x)$ is continuous ans bounded.
For a positive number $\varepsilon>0$,
put
$$
\varphi_\varepsilon(x):=\frac{1}{\varepsilon}
\int_{\mathbf R}K\(\frac{t-x}{\varepsilon}\)\varphi(t)dt.
$$
Then, for any $x\in\mathbf R$, we have
$
\varphi(x) = \lim_{\varepsilon\rightarrow + 0}\varphi_\varepsilon(x).
$
\end{lemma}
\begin{proof}
Since $\varphi_\varepsilon(x)$ is decomposed as
\begin{align*}
\varphi_\varepsilon(x)
&=\varphi(x)
  \int_{-\infty}^\infty K\(\frac{x-t}{\varepsilon}\)\frac{dt}{\varepsilon}
 -\int_{-\infty}^\infty \(\varphi(x)-\varphi(t)\)
     K\(\frac{x-t}{\varepsilon}\)\frac{dt}{\varepsilon}\\
&=\varphi(x)
  -\int_{-\infty}^\infty \(\varphi(x)-\varphi(t)\)
     K\(\frac{x-t}{\varepsilon}\)\frac{dt}{\varepsilon},
\end{align*}
it is enough to show that the second term converges to zero.
We decompose the integral domain into
$|x-t|\geq\varepsilon$ and $|x-t|<\varepsilon$.
On the first domain, we have
\begin{align*}
\left|
\int_{|x-t|\geq\varepsilon}
     \(\varphi(x)-\varphi(t)\)
     K\(\frac{x-t}{\varepsilon}\)\frac{dt}{\varepsilon}
\right|
&\leq
\(2\sup_x |\varphi(x)| \)
\int_{|x-t|\geq\varepsilon}
     \left| K\(\frac{x-t}{\varepsilon}\) \right|
     \frac{dt}{\varepsilon}\\
&=
\(2\sup_x|\varphi(x)|\)
\int_{|s|\geq 1/\sqrt{\varepsilon}}
     \left| K\(s\)\right| ds.
\end{align*}
Note that $\sup_x|\varphi(x)|$ is finite since $\varphi(x)$ is bounded.
Since $K(x)$ is absolutely integrable,
$\int_{|s|\geq 1/\sqrt{\varepsilon}} \left| K\(s\)\right| ds$
converges to $0$ as $\varepsilon\rightarrow+0$.
Consequenty, we have
$$
\left|
\int_{|x-t|\geq\varepsilon}
     \(\varphi(x)-\varphi(t)\)
     K\(\frac{x-t}{\varepsilon}\)\frac{dt}{\varepsilon}
\right|
\rightarrow 0
\quad (\varepsilon\rightarrow +0).
$$
On the second domain, we have
\begin{align*}
\left|
\int_{|x-t|<\varepsilon}
     \(\varphi(x)-\varphi(t)\)
     K\(\frac{x-t}{\varepsilon}\)\frac{dt}{\varepsilon}
\right|
&\leq
\(\sup_{|x-t|<\varepsilon} |\varphi(x)-\varphi(t)| \)
\int_{|s|\geq 1/\sqrt{\varepsilon}}
     \left| K\(s\)\right| ds\\
&\leq
\(\sup_{|x-t|<\varepsilon} |\varphi(x)-\varphi(t)| \)
     \int_{-\infty}^\infty\left| K\(s\)\right| ds.
\end{align*}
Since $\varphi(x)$ is continuous,
$\sup_{|x-t|<\varepsilon} |\varphi(x)-\varphi(t)|$
goes to zero as $\varepsilon\rightarrow +0$.
Consequently, we have
$$
\left|
\int_{|x-t|<\varepsilon}
     \(\varphi(x)-\varphi(t)\)
     K\(\frac{x-t}{\varepsilon}\)\frac{dt}{\varepsilon}
\right|
\rightarrow 0
\quad (\varepsilon\rightarrow +0).
$$
Therefore, we have
$
\varphi(x) = \lim_{\varepsilon\rightarrow + 0}\varphi_\varepsilon(x).
$
\end{proof}

\begin{proof}[Proof of Proposition \ref{a7Jun2018}]
Fix the continuous point $x$ of $f$.
Decompose $f$ into a sum of two continuous functions $f_1$ and $f_2$
where $f_1$ is support compact
and $f_2$ equals to zero on a neighborhood of $x$.
Put $\varphi_j(t):=\int_{-\infty}^\infty e^{ity}f_j(y)dy\,(j=1,2)$.
By the Fubini's theorem, we have
\begin{align*}
&
\frac{1}{2\pi}\(
 \int_{-\infty}^0 e^{-i(x+i\varepsilon)t}\varphi_1(t)dt
+\int_0^\infty   e^{-i(x-i\varepsilon)t}\varphi_1(t)dt
\)\\
&=
\frac{1}{2\pi}\(
 \int_{-\infty}^\infty f_1(y)
    \int_{-\infty}^0 e^{i(y-x-i\varepsilon)t}
 dtdy
+\int_{-\infty}^\infty f_1(y)
    \int_0^\infty   e^{i(y-x+i\varepsilon)t}dtdy
\)\\
&=
 \int_{-\infty}^\infty f_1(y)
    \frac{1}{\pi} 
    \frac{\varepsilon}{(y-x)^2+\varepsilon^2}
  dy
\end{align*}
Apply Lemma \ref{b8Jun2018} with $K(x)=\frac{1}{\pi(x^2+1)}$,
then we have
$$
 \int_{-\infty}^\infty f_1(y)
    \frac{1}{\pi} 
    \frac{\varepsilon}{(y-x)^2+\varepsilon^2}
  dy
\rightarrow f_1(x)
\quad (\varepsilon\rightarrow +0).
$$
Analogously, we have
\begin{align*}
\frac{1}{2\pi}\(
 \int_{-\infty}^0 e^{-i(x+i\varepsilon)t}\varphi_2(t)dt
+\int_0^\infty   e^{-i(x-i\varepsilon)t}\varphi_2(t)dt
\)
&=
\frac{1}{2\pi}
 \int_{-\infty}^\infty f_2(y)
    \(\frac{1}{i(y-x-i\varepsilon)}-\frac{1}{i(y-x+i\varepsilon)}\)
  dy
\end{align*}
Since $f_2$ equals to zero on a neighborhood of $x$,
for sufficiently small $\delta>0$, 
the above integral equals to 
\begin{align*}
&
\frac{1}{2\pi}
 \int_{|y-x|>\delta}
    f_2(y)
    \(\frac{1}{i(y-x-i\varepsilon)}-\frac{1}{i(y-x+i\varepsilon)}\)
  dy
\end{align*}
By the Lebesgue's dominated convergence theorem, 
this integral converges to $0$
as $\varepsilon\rightarrow+0$.
Hence, we have
\begin{align*}
&
\frac{1}{2\pi}\lim_{\varepsilon\rightarrow +0}\(
 \int_{-\infty}^0 e^{-i(x+i\varepsilon)t}\varphi(t)dt
+\int_0^\infty   e^{-i(x-i\varepsilon)t}\varphi(t)dt
\)\\
&=\sum_{j=1}^2
\frac{1}{2\pi}\lim_{\varepsilon\rightarrow +0}\(
 \int_{-\infty}^0 e^{-i(x+i\varepsilon)t}\varphi_j(t)dt
+\int_0^\infty   e^{-i(x-i\varepsilon)t}\varphi_j(t)dt
\)
=f_1(x) = f(x).
\end{align*}
\end{proof}

\section{Perspective of Hyperfunction}\label{sec:hyperfunction}
In this section,
we briefly review the theory of hyperfunctions
and justify the intuitive discussion
in Section \ref{intro} from the view point of hyperfunctions.

In the first, we review the theory of hyperfunctions of
a single variable based on \cite{sato1958} and \cite{kaneko1988}.
Let $\mathcal O$ be the sheaf of holomorphic functions on
the complex plane $\mathbf C$.
The sheaf of the hyperfunctions $\mathcal B$ of a single variable is defined as
the $1$-th 
derived sheaf $\mathcal H_{\mathbf R}^1(\mathcal O)$ of $\mathcal O$.
The global section $\mathcal B(\mathbf R)$ is
the inductive limit
$\varinjlim_{U\supset\mathbf R}\mathcal O(U\backslash\mathbf R)/O(U)$
with respect to the family of complex neighborhoods $U\supset\mathbf R$.
The elements of $\mathcal B(\mathbf R)$ are called hyperfunctions on $\mathbf R$.
Any hyperfunction on $\mathbf R$ can be represented as
a equivalent class $[F(z)]$ with a representative
$F(z)\in\mathcal O(U\backslash\mathbf R)$.

Let $L_{1,loc}(\mathbf R)$ be
the spaces of locally integrable functions on $\mathbf R$.
There is a natural embedding $\iota$
from $L_{1,loc}(\mathbf R)$ to $\mathcal B(\mathbf R)$.
When an integrable function $f\in L_1(\mathbf R)$ satisfies some condition,
we can take a representative $F(z)$ of $\iota(f)$
as a holomorphic function on $\mathbf C\backslash\mathbf R$.
\begin{proposition}[M.~Sato]\label{sato-1}
Let $f(x)$ be an integrable function on $\mathbf R$
such that the integral $\int_{-\infty}^\infty (x^2+1)^{-1}f(x)dx $ is finite.
Take a constant $\alpha\in\C-\R$, and put 
\begin{align*}
F(z)
&=\frac{1}{2\pi i}\int_{-\infty}^\infty f(x)
\(\frac{1}{x-z}-\frac{1}{x-\alpha}\)dx
\quad (z\in\C\backslash\R).
\end{align*}
Then we have $\iota(f)=[F]$.
\end{proposition}
\begin{proof}
See, \cite{sato1958}.
\end{proof}

We call a measure $\mu$ on $\mathbf R$ to be locally integrable
when $\mu(K)$ is finite for any compact set $K\subset\mathbf R$.
We denote by $M_{loc}$ the space of the locally integrable measures
on $\mathbf R$.
Analogous to the case of $L_{1,loc}(\mathbf R)$,
there is a natural embedding $\iota$ from $M_{loc}$ to $\mathcal B(\mathbf R)$,
and the following proposition holds:
\begin{proposition}[M.~Sato]\label{sato-2}
Let $\mu$ be a measure on $\mathbf R$ such that
the integral $\int_{\mathbf R} (x^2+1)^{-1}\mu(dx) $ is finite.
Take a constant $\alpha\in\C-\R$, and put 
\begin{align*}
F(z)
&=\frac{1}{2\pi i}\int_{\mathbf R}
\(\frac{1}{x-z}-\frac{1}{x-\alpha}\)\mu(dx)
\quad (z\in\C\backslash\R)
\end{align*}
Then we have $\iota(\mu)=[F]$.
\end{proposition}

Let $I$ be a neighborhood of $0\in\mathbf R$,
and put $I^*:= I\backslash \{0\}$.
A function $F(z)$ holomorphic on a tubular domain
$
\mathbf R+\I I^*
:=\left\{
z\in\mathbf C \vert \Im z \in I^*
\right\}
$
is said to be slowly increasing
if for any compact subset $K\subset I^*$ and any $\varepsilon>0$,
there exists $C>0$ such that
$$
|F(z)| \leq C e^{\varepsilon |\Re z|} 
$$
uniformly for $z\in \mathbf R+\I K$.
A hyperfunction $f$ on $\mathbf R$ is said to be slowly increasing 
if we can take a slowly increasing function $F(z)$ as
its representative, i.e.,
there exists a slowly increasing function $F(z)$
such that $f=[F(z)]$.
Put
\begin{align*}
\chi_+(z) &:= \frac{e^z}{1+e^z},&
\chi_-(z) &:= \frac{e^{-z}}{1+e^{-z}}.
\end{align*}
The following definition is a specialization into the case of a single variable
of the general definition in \cite{kaneko1988}.
\begin{definition}\label{def:Fourier}
Let a slowly increasing function $F(z)$ 
on a tubular domain $\mathbf R+\I I^*$ be
a representative of a slowly increasing hyperfunction
$f\in\mathcal B(\mathbf R)$.
Take a positive number $y\in I^*$.
We define the Fourier transform of $f=[F(z)]$
as $\mathcal Ff=[G(w)]$ where
$$
G(w)=
\begin{cases}
 \int_{\Im z=+y} e^{-\I wz}\chi_-(z)F(z) dz
-\int_{\Im z=-y} e^{-\I wz}\chi_-(z)F(z) dz
& (\Im w > 0)\\
-\int_{\Im z=+y} e^{-\I wz}\chi_+(z)F(z) dz
+\int_{\Im z=-y} e^{-\I wz}\chi_+(z)F(z) dz
& (\Im w < 0).
\end{cases}
$$
Here,
for $y\in\mathbf R$ and a function $\varphi(z)$ with a complex variable $z$,
we denote the integral
$
\int_{\mathbf R} \varphi(x+\I y) dx
$
by
$
\int_{\Im z = y} \varphi(z)dz.
$
\end{definition}

In the second, we give and show an equation
corresponding to the formal equation \eqref{formal}.
Let $X$ be a random variable.
The distribution $\mu_X$ of $X$ and 
the characteristic function $\varphi_X(t)=\E{e^{\I tX}}$ of $X$
can be regarded as hyperfunctions.
We denote them by $\iota\mu_x$ and $\iota\varphi_X$ respectively.
The following equation corresponds to the formal equation
\eqref{formal}.
\begin{equation}\label{2}
\mathcal F(\iota \varphi_X) = 2\pi \iota \mu_X.
\end{equation}
Since the characteristic function can be regarded as
the inverse Fourier transformation of the distribution
under a suitable condition,
the equation \eqref{2} is an inversion formula.

We calculate defining functions of $\iota\mu_X$ and $\iota\varphi_X$.
Let $\alpha\in\C\backslash\R$ be a constant.
Since $\mu_X$ is a probability measure on $\mathbf R$,
we have
$$
\left|\int_{\mathbf R}\frac{1}{(x^2+1)}\mu_X(dx) \right|
\leq \int_{\mathbf R}\frac{1}{|x^2+1|}\mu_X(dx)
\leq \int_{\mathbf R}\mu_X(dx)
= 1
<\infty
$$
and
$$
\left|\int_{\mathbf R} \frac{1}{x-\alpha}\mu_X(dx)\right|
\leq \int_{\mathbf R} \frac{1}{|x-\alpha|}\mu_X(dx)
\leq \int_{\mathbf R} \frac{1}{|\Im\alpha|}\mu_X(dx)
= \frac{1}{|\Im\alpha|}
<\infty.
$$
By Proposition \ref{sato-2}, we have
\begin{equation}\label{20}
\iota\mu_X
=\left[\frac{1}{2\pi i}\int_{\mathbf R}
\(\frac{1}{x-z}-\frac{1}{x-\alpha}\)\mu_X(dx)
\right]
=\left[\frac{1}{2\pi i}\int_{\mathbf R}\frac{1}{x-z}\mu_X(dx)\right].
\end{equation}
Since the estimation 
$
\left|
 \int_{-\infty}^\infty \varphi_X(t)(t^2+1)^{-1}dt
\right|
<\infty
$
holds, we can apply Proposition \ref{sato-1} to $\varphi_X(t)$.
\ifverbose
In fact, we have
\begin{align*}
\left|
 \int_{-\infty}^\infty \frac{\varphi_X(t)}{t^2+1}dt
\right|
&\leq
\int_{-\infty}^\infty
 \left|\frac{\varphi_X(t)}{t^2+1}\right|
dt
\leq \int_{-\infty}^\infty \frac{1}{t^2+1}dt
= \left[\arctan(t)\right]_{-\infty}^\infty
= \pi.
\end{align*}
\fi
Hence, we have
$$
\iota\varphi_X=\left[
\frac{1}{2\pi\I}
\int_{-\infty}^\infty \varphi_X(t)\left(
  \frac{1}{t-w}-\frac{1}{t-\alpha}
\right)dt
\right].
$$

We calculate the Fourier transformation of $\iota\varphi_X$.
\begin{lemma}\label{1}
For a random variable $X$,
the hyperfunction $\iota\varphi_X$ corresponding to
the characteristic function $\varphi_X$ of $X$
is slowly increasing.
\end{lemma}
\begin{proof}
For simplicity, we assume $\alpha=\sqrt{-1}$ without loss of generality.
The absolute value of the representative of $\iota\varphi_X$
satisfies the inequality
\begin{equation}\label{15}
\left|
\frac{1}{2\pi\I}
\int_{-\infty}^\infty \varphi_X(t)\left(
  \frac{1}{t-w}-\frac{1}{t-\alpha}
\right)dt
\right|
\leq
\frac{|w-\alpha|}{2\pi}
\int_{-\infty}^\infty\frac{dt}{|t-w||t-\alpha|}.
\end{equation}
Put $c:=|\Re w|$, then the inequalities
\begin{align*}
  \int_{c+1}^\infty\frac{dt}{|t-w||t-\alpha|}
&\leq 
  \int_{c+1}^\infty\frac{dt}{(t-c)^2}
=
  1, &
  \int_{-\infty}^{-(c+1)}\frac{dt}{|t-w||t-\alpha|}
&\leq 
  \int_{-\infty}^{-c-1}\frac{dt}{(t+c)^2}
=
  1,\\
  \int_{-c-1}^{c+1}\frac{dt}{|t-w||t-\alpha|}
&\leq 
  \frac{2(c+1)}{|\Im w|},
\end{align*}
imply that the right hand side of \eqref{15} is bounded above by
$$
\frac{|w-\alpha|}{\pi}
\left(
1+  \frac{|\Re w|+1}{|\Im w|}
\right).
$$
Hence, $\iota\varphi_X$ is a slowly increasing hyperfunction.
\end{proof}
In order to calculate an explicit form of a defining function of
the Fourier transformation of $\iota\varphi_X$,
we decompose the function $\varphi_X(t)$ as
$$
\varphi_X(t)= (1-H(t))\varphi_X(t) +H(t)\varphi_X(t)
$$
where $H(t)$ is the Heaviside function.
Embedding the both sides of the above equation into
the global section $\hf(\R)$ of the sheaf of hyperfunctions, we have
$$
\iota\varphi_X = \iota((1-H)\varphi_X) + \iota(H \varphi_X)
$$
By Proposition \ref{sato-1}, defining functions of
$\iota((1-H)\varphi_X)$ and $\iota(H \varphi_X)$ are given by 
\begin{align*}
\varphi_1(w)&:=
\frac{1}{2\pi\I}
\int_{-\infty}^0 \varphi_X(t)\left(
  \frac{1}{t-w}-\frac{1}{t-\alpha}
\right)dt
\quad (w\in\C-\R_{\leq 0}),\\
\varphi_2(w)&:=
\frac{1}{2\pi\I}
\int_0^\infty \varphi_X(t)\left(
  \frac{1}{t-w}-\frac{1}{t-\alpha}
\right)dt
\quad (w\in\C-\R_{\geq 0})
\end{align*}
respectively.
Then we have
$ 
\iota\varphi_X = [\varphi_1(w)]+[\varphi_2(w)].
$ 
Analogous to Lemma \ref{1},
we can show that
both $[\varphi_1(w)]$ and $[\varphi_2(w)]$ are
slowly increasing hyperfunctions.

For a function $F(z)$ with a complex variable ,
we denote by $\int_{+\infty}^{(0+)} F(z) dz$
the integral along the path described in Figure \ref{contour-a}.
We also denote by $\int_{-\infty}^{(0+)} F(z) dz$
the integral along the path described in Figure \ref{contour-b}.

\begin{figure}[htbp]
 \begin{minipage}{0.45\hsize}
  \center
  \ifincludefigure
  \includegraphics[width=0.95\hsize]{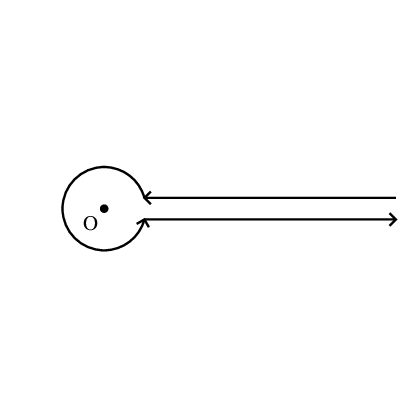}\\
  \fi
  \caption{Contour of the path of $\int_{+\infty}^{(0+)}$}
  \label{contour-a}
 \end{minipage}\hfill 
 \begin{minipage}{0.45\hsize}
  \ifincludefigure
  \includegraphics[width=0.95\hsize]{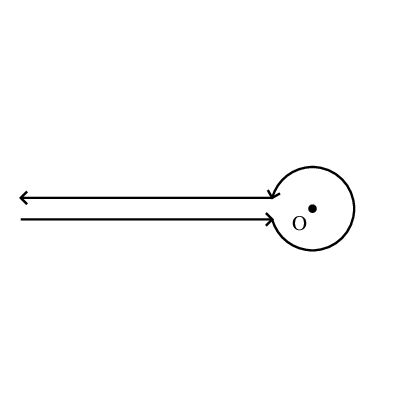}\\
  \fi
  \caption{ Contour of the path of $\int_{-\infty}^{(0+)}$}
  \label{contour-b}
 \end{minipage}
\end{figure}

We denote by $\H$ the upper half-plane $\{z\in\mathbf C\vert \Im z >0\}$.
\begin{lemma}
The holomorphic functions 
\begin{align*}
\psi_1(z) &:=
\begin{cases}
-\int_{-\infty}^{(0+)}e^{-\I zw}\varphi_1(w)dw &(z\in\H),\\
0 &(z\in-\H),
\end{cases}\\
\psi_2(z) &:=
\begin{cases}
0&(z\in\H),\\
\int_{+\infty}^{(0+)}e^{-\I zw}\varphi_2(w)dw & (z\in-\H),
\end{cases}
\end{align*}
are defining functions of the Fourier transformations of
$[\varphi_1(w)]$ and $[\varphi_2(w)]$ respectively, i.e.,
we have
$\mathcal F[\varphi_k(w)] = [\psi_k(z)]\,(k=1,2)$.
\end{lemma}
\begin{proof}
By Definition \ref{def:Fourier}, a defining function of $\mathcal F[\varphi_1(w)]$
is given by
$$
\tilde\psi_1(z)=
\begin{cases}
 \int_{\Im w=+y} e^{-\I zw}\chi_-(w)\varphi_1(w) dw
-\int_{\Im w=-y} e^{-\I zw}\chi_-(w)\varphi_1(w) dw
& (\Im z > 0)\\
-\int_{\Im w=+y} e^{-\I zw}\chi_+(w)\varphi_1(w) dw
+\int_{\Im w=-y} e^{-\I zw}\chi_+(w)\varphi_1(w) dw
& (\Im z < 0).
\end{cases}
$$
Since $\varphi_1$ is holomorphic on $\mathbf C-\mathbf R_{\leq 0}$,
we have
$$
\tilde\psi_1(z)=
\begin{cases}
 -\int_{-\infty}^{(0+)}  e^{-\I zw}\chi_-(w)\varphi_1(w) dw
& (\Im z > 0)\\
 \int_{-\infty}^{(0+)}  e^{-\I zw}\chi_+(w)\varphi_1(w) dw
& (\Im z < 0).
\end{cases}
$$
Note that the integral
$ \int_{-\infty}^{(0+)}  e^{-\I zw}\chi_+(w)\varphi_1(w) dw$
defines a holomorphic function on a neighborhood of $\mathbf R$.
Hence,
$$
\psi_1(z)=\tilde\psi_1(z)
-\int_{-\infty}^{(0+)}  e^{-\I zw}\chi_+(w)\varphi_1(w) dw
$$
is a defining function of $\mathcal F[\varphi_1(w)]$ also.

We can show $\mathcal F[\varphi_2(w)]=[\psi_2(z)]$ analogously.
\end{proof}
\begin{lemma}\label{4}
Put
$$
\psi(z) =
\begin{cases}
\int_{-\infty}^0 e^{-\I zt}\varphi_X(t)dt & (z\in\H),\\
-\int_0^\infty e^{-\I zt}\varphi_X(t)dt & (z\in-\H).
\end{cases}
$$
Then $\psi(z)$ is a defining function of $\mathcal F\iota\varphi_X$.
\end{lemma}
\begin{proof}
It is enough to calculate the right hand side of
the equation $\mathcal F\iota\varphi_X=[\psi_1+\psi_2]$.

For $z\in-\H$, we have
\begin{align*}
\psi_2(z)
&=\int_{+\infty}^{(0+)}e^{-\I zw}
\frac{1}{2\pi\I}
\int_0^\infty \varphi_X(t)\left(
  \frac{1}{t-w}-\frac{1}{t-\alpha}
\right)dt
dw\\
&=\frac{1}{2\pi\I}\int_{+\infty}^{(0+)}\int_0^\infty
  e^{-\I zw}\varphi_X(t)\left(
  \frac{1}{t-w}-\frac{1}{t-\alpha}
\right)dt dw
\end{align*}
By the assumption $z\in-\H$, we have
$$
\int_{+\infty}^{(0+)}\int_0^\infty\left|
  e^{-\I zw}\varphi_X(t)\left(
  \frac{1}{t-w}-\frac{1}{t-\alpha}
\right)\right| dt dw 
<\infty.
$$ 
By the Fubini-Tonelli theorem and Cauchy's integral formula,
$\psi_2(z)$ equals to 
\begin{align*}
&
\frac{1}{2\pi\I}\int_0^\infty\int_{+\infty}^{(0+)}
  e^{-\I zw}\varphi_X(t)\left(
  \frac{1}{t-w}-\frac{1}{t-\alpha}
\right)dw dt \\
&=
-\frac{1}{2\pi\I}\int_0^\infty\varphi_X(t)\int_{+\infty}^{(0+)}
  e^{-\I zw}\left(
  \frac{1}{w-t}+\frac{1}{t-\alpha}
\right)dw dt 
=
-\int_0^\infty
  e^{-\I zt}\varphi_X(t)
  dt.
\end{align*}
For $z\in\H$,  we can show
$\psi_1(z)=\int_{-\infty}^0 e^{-\I zt}\varphi_X(t)dt$
analogously.
\end{proof}
\begin{lemma}\label{21}
For $z\in\mathbf C\backslash\mathbf R$,
the following equation holds:
\begin{equation}\label{18}
\psi(z)=
\frac{1}{i}\int_{\mathbf R}\frac{1}{x-z}\mu_X(dx).
\end{equation}
\end{lemma}
\begin{proof}
For $z\in-\H$, we have
\begin{equation}\label{16}
-\int_0^\infty e^{-\I zt}\varphi_X(t)dt
=
-\int_0^\infty\int_{\R} e^{-\I(z-x)t}\mu_X(dx)dt
\end{equation}
by the definition of the characteristic function.
Since $\Im z$ is negative, we have
\begin{align*}
\int_0^\infty\int_{\R} \left|e^{-\I(z-x)t}\right|\mu_X(dx)dt
&=
\int_0^\infty e^{\Im z t}dt
\leq
-\frac{1}{\Im z}
<\infty.
\end{align*}
By the Fubini-Tonelli theorem, the right hand side of \eqref{16}
equals to 
\begin{align*}
-\int_{\R}\int_0^\infty e^{-\I(z-x)t}dt\mu_X(dx)
&=
\int_{\R}\frac{1}{-\I(z-x)}\mu_X(dx)
=
\frac{1}{\I}\int_{\R}\frac{1}{x-z}\mu_X(dx).
\end{align*}
For $z\in\H$,  we can show the equation \eqref{18} analogously.
\end{proof}

\begin{theorem}\label{5}
For a random variable $X$, equation \eqref{2} holds.
\end{theorem}
\begin{proof}
By Lemma \ref{21} and equation \eqref{20}, we have
\begin{align*}
\mathcal F\iota\varphi_X
&=[\psi(z)]
=\left[\frac{1}{\I}\int_{\mathbf R}\frac{1}{x-z}\mu_X(dx)\right]
=2\pi\iota\mu_X.
\end{align*}
\end{proof}

\begin{proof}[Other Proof of Proposition \ref{a7Jun2018}]
Let $X$ be a random variable whose probability density function is $f(x)$.
Then, $\varphi(t)$ equals to the characteristic function of $X$.
By Lemma \ref{1} and Lemma \ref{4},
the hyperfunction $\iota\varphi$ corresponding to $\varphi(t)$
is a Fourier hyperfunction,
and a defining function of $\mathcal F\iota \varphi$ is
$$
\psi(z) :=
\begin{cases}
\int_{-\infty}^0 e^{-\I zt}\varphi(t) dt & (z\in\H),\\
-\int_0^\infty e^{-\I zt}\varphi(t) dt & (z\in-\H).
\end{cases}
$$
By Theorem \ref{5}, the corresponding hyperfunction $\iota f$ of $f$
equals to 
$(2\pi)^{-1}\mathcal F\iota\varphi$.
For any continuous point $x$ of $f$, 
the boundary-value representation \cite[Theorem 1.3.12]{kaneko1988} implies
$$
f(x) =
\frac{1}{2\pi}
\lim_{y\rightarrow +0}\left(
\int_{-\infty}^0 e^{-\I (x+\I y)t}\varphi(t) dt 
+\int_0^\infty e^{-\I (x-\I y)t}\varphi(t) dt 
\right)
$$
\end{proof}

\section{Sum of $r$-th Power of the Normal Variables}\label{PowerSum}
Fix a positive integer $r\geq 3$.
Let $X_1,\dots, X_n$ be 
the independent, identically and normally distributed random variables, 
and put $X:=\sum_{k=1}^n X_k^r$.
Let $\varphi(t)$ be the characteristic function of $X_1^r$, i.e.,
put
\begin{equation}\label{6}
\varphi(t) = \E{e^{\I t X_1^r}}
=\int_{\R} e^{\I tx^r}\frac{1}{\sqrt{2\pi}}e^{-\frac{1}{2}x^2}dx.
\end{equation}
Then, the characteristic function $\varphi_X$
of the powered sum $X$ equals to $\varphi(t)^n$.
By Proposition \ref{a7Jun2018}, we have
\begin{equation}\label{14}
f_X(x) =
\frac{1}{2\pi}
\lim_{y\rightarrow +0}\left(
\int_{-\infty}^0 e^{-\I (x+\I y)t}\varphi(t)^n dt 
+\int_0^\infty e^{-\I (x-\I y)t}\varphi(t)^n dt 
\right)
\end{equation}
for any continuous point $x\in\R$ of $f_X$.

In order to calculate the limit in the right hand side of \eqref{14},
we discuss on the analytic continuations of
the each term of \eqref{19}.

In the first, we consider the characteristic function in \eqref{6}.
The characteristic function can be decomposed as
$$
\varphi(t) 
=\frac{1}{\sqrt{2\pi}}\int_0^\infty\exp\left(\I tx^r-\frac{1}{2}x^2\right)dx
+\frac{1}{\sqrt{2\pi}}\int_{-\infty}^0\exp\left(\I tx^r-\frac{1}{2}x^2\right)dx.
$$
We are interested in the analytic continuation of each term
in the right hand side.
In order to calculate them, 
we consider the following general form of integral:
\begin{equation}\label{7}
\varphi_\theta(w)
:=\frac{1}{\sqrt{2\pi}}\int_{\gamma_\theta}
  \exp\left(\I wz^r-\frac{1}{2}z^2\right)dz
\quad (w\in\C),
\end{equation}
where the integral path $\gamma_\theta$ with parameter $\theta\in\R$ is defined by
\begin{equation}\label{8}
\gamma_\theta(t) := te^{\I\theta}
\quad (0<t<\infty).
\end{equation}
Note that we have
$\varphi(t)=\varphi_0(t)-\varphi_\pi(t)$ for $t\in\R$.

For $z\in\mathbf C$, put
$$
z\H :=\left\{zw \vert w\in\H\right\}.
$$
Examining the convergence region of the integral \eqref{7}
for given parameter $\theta\in\R$,
we obtain the following lemma:
\begin{lemma}
Suppose that $d\geq 3$.
For given parameter $\theta\in\R$, the integral \eqref{7} converges on
$e^{-\I\theta r}\H$.
\end{lemma}
\begin{proof}
Take any point $w\in e^{-\I\theta r}\H$.
By the straight forward calculation, we have
\begin{align*}
\left|\varphi_\theta(w)\right|
&\leq \frac{1}{\sqrt{2\pi}}
\int_0^\infty\left|
  \exp\left(\I we^{\I\theta r}t^r-\frac{1}{2}e^{2\I\theta}t^2\right)e^{\I\theta}
\right|dt\\
&= \frac{1}{\sqrt{2\pi}}
\int_0^\infty
  \exp\left(
     -\Im(w e^{\I\theta r}) t^r 
    -\frac{\cos(2\theta)}{2}t^2
  \right)
dt
\end{align*}
Since assumption $we^{\I\theta r}\in\H$ implies $\Im(we^{\I\theta r})>0$,
there exist $t_0>1$ and $\varepsilon>0$ such that
$$
-\Im(we^{\I\theta r}) - \frac{\cos(2\theta)}{2t^{r-2}}
< -\varepsilon
$$
holds for any $t>t_0$.
Hence, we have
\begin{align*}
&
\int_{t_0}^\infty
  \exp\left(
    -\Im(we^{\I\theta r})t^r
    -\frac{\cos(2\theta)}{2}t^2
  \right)
dt\\
&\leq 
\int_{t_0}^\infty
  \exp\left(
    -\varepsilon t^r
  \right)
dt
\leq
\int_{t_0}^\infty
  r\cdot t^{r-1}
  \exp\left(
    -\varepsilon t^r
  \right)
dt
=
\frac{\exp(-\varepsilon t_0)}{\varepsilon}.
\end{align*}
Consequently, the integral \eqref{7} converges on $e^{-\I\theta r}\H$.
\end{proof}

\begin{lemma}
Suppose that $d\geq 3$.
If $\cos 2\theta>0$ holds, then the integral \eqref{7} converges on
$e^{-\I\theta r}\bar\H$.
Here, $\bar\H$ denotes the closure of $\H$.
\end{lemma}
\begin{proof}
For any point $w\in e^{-\I\theta r}\bar\H$,
we have $\Im(we^{\I\theta r})\geq 0$,
This implies
\begin{align*}
\left|\varphi_\theta(w)\right|
&\leq \frac{1}{\sqrt{2\pi}}
\int_0^\infty
  \exp\left(
    -\Im(w e^{\I\theta r}) t^r 
    -\frac{\cos(2\theta)}{2}t^2
  \right)
dt\\
&\leq \frac{1}{\sqrt{2\pi}}
\int_0^\infty
  \exp\left(
    -\frac{\cos(2\theta)}{2}t^2
  \right)
dt
<\infty
\end{align*}
\end{proof}

\begin{example}
Let $r=3$.
Figures \ref{fig1-a} and \ref{fig1-b} show
the contour of the integral path $\gamma_\theta$
and the region $\exp\(-\I\theta r\)\H$
for $\theta=0$ respectively.
Figures \ref{fig2-a} and \ref{fig2-b} show
them for $\theta=\pi/9$.
\begin{figure}[h]
 \begin{minipage}{0.45\hsize}
  \ifincludefigure
   \center
   \includegraphics[width=0.95\hsize]{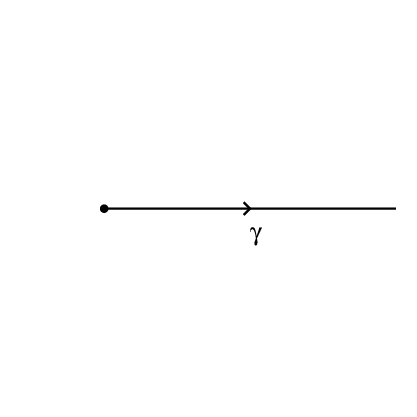}
   \caption{The integral path for $\theta=0$}
   \label{fig1-a}
  \fi
 \end{minipage}
 \begin{minipage}{0.45\hsize}
  \center
  \ifincludefigure
  \includegraphics[width=0.95\hsize]{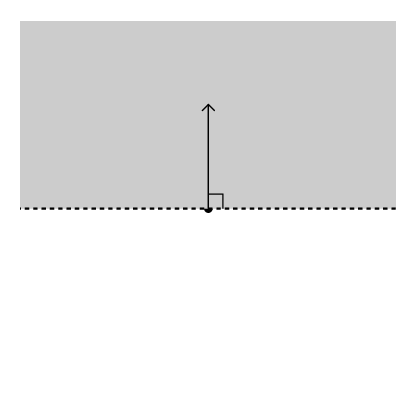}
  \fi
  \caption{The convergence region for $\theta=0$}
  \label{fig1-b}
 \end{minipage}
\end{figure}
\begin{figure}
 \begin{minipage}{0.45\hsize}
  \center
  \ifincludefigure
   \includegraphics[width=0.95\hsize]{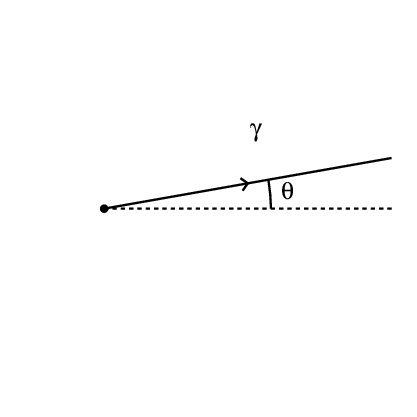}
  \fi
  \caption{The integral path for $\theta=\pi/9$}
  \label{fig2-a}
 \end{minipage}
 \begin{minipage}{0.45\hsize}
  \center
  \ifincludefigure
   \includegraphics[width=0.95\hsize]{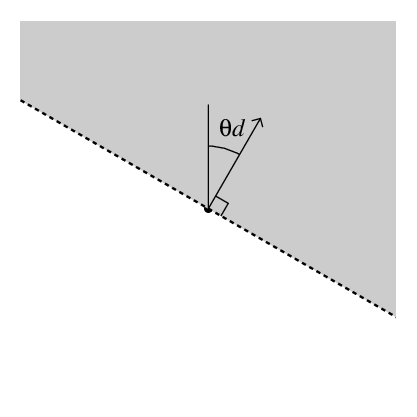}
  \fi
 \caption{The convergence region for $\theta=\pi/9$}
 \label{fig2-b}
 \end{minipage}
\end{figure}
\end{example}

\begin{lemma}
If $|\theta-\theta'|<\pi/d$, then
$\varphi_\theta(w) = \varphi_{\theta'}(w)$
holds for $w\in e^{-\I\theta r}\H\cap e^{-\I\theta'r}\H$.
\end{lemma}
\begin{proof}
Put
\begin{align*}
\gamma_{1,R}(t) &:= te^{\I\theta} \,(0 \leq t \leq R), &
\gamma_{2,R}(t) &:= te^{\I\theta'} \,(0 \leq t \leq R),&
\gamma_{3,R}(\theta'') &:= Re^{\I\theta''} \,(\theta \leq \theta'' \leq \theta').
\end{align*}
By the Cauchy's integral formula, we have
\begin{equation}\label{13}
\int_{\gamma_{1,R}} \Phi dz
-\int_{\gamma_{2,R}} \Phi dz
+\int_{\gamma_{3,R}} \Phi dz
=0.
\end{equation}
Here, we put
$
\Phi = (2\pi)^{-1/2} \exp\left(\I wz^r-z^2/2\right).
$
By some calculations, we have that there exist $M>0$ and $\varepsilon>0$
such that 
$$
\left| \int_{\gamma_{3,R}} \Phi dz \right|
\leq M R\exp\(-\varepsilon R^r\).
$$
Taking the limit of \eqref{13} as $R$ approaches inifinity,
we have
$
\varphi_\theta(z) - \varphi_{\theta'}(z) =0.
$
\end{proof}

In the second, we disscuss on the analytic continuation
of the each term of \eqref{19}.
Let us consider the following integral:
\begin{equation}\label{17}
\psi(z;\theta,\theta',\phi):=
\int_{\gamma_\phi}
 e^{-\I w z}\left(\varphi_\theta(w)-\varphi_{\theta'}(w)\right)^n
dw.
\end{equation}
where $\theta,\theta',\phi\in\R$, $z\in\C$, and 
the integral path $\gamma_\phi$ is defined by \eqref{8}.
\begin{lemma}\label{d8Jun2018}
Suppose $\Im(e^{\I(\phi+\theta r)})\geq 0$ and $\Im(e^{\I(\phi+\theta'r)})\geq 0$.
Then the integral \eqref{17} converges for $z\in -\exp(-\I\phi)H$.
\end{lemma}
\begin{proof}
By the binomial expansion, the integral \eqref{17} equals to
$$
\sum_{k=0}^n\binom{n}{k}(-1)^{n-k}
\int_{\gamma_\phi} 
  e^{-\I w z}
  \varphi_\theta(w)^k\varphi_{\theta'}(w)^{n-k}
dw.
$$
With some calculation, the each term of the above expression can be written as 
$$
\int_{\gamma_\phi} 
  e^{-\I w z}
  \varphi_\theta(w)^k\varphi_{\theta'}(w)^{n-k}
dw
=
\frac{1}{(2\pi)^{n/2}}\int_0^\infty\int_{\R_{>0}^n}
 \exp\(-h_1(z,t) s+h_2(t)\)
dt ds,
$$
where we put $t:=(t_1,\dots,t_n)^\top$, $dt:=dt_1\dots dt_n$,
\begin{align*}
h_1(z,t)
&=\I e^{\I\phi}
\left(
   z
  -e^{\I\theta r}\sum_{\ell=1  }^kt_\ell^r
  -e^{\I\theta'r}\sum_{\ell=k+1}^nt_\ell^r
\right), &
h_2(t)
&=
-\frac{1}{2}
\left(
   e^{2\I\theta }\sum_{\ell=1  }^kt_\ell^2
  +e^{2\I\theta'}\sum_{\ell=k+1}^nt_\ell^2
\right).
\end{align*}
By the Fubini-Tonelli theorem, 
it is enough to consider the convergence of the following integral:
\begin{equation}\label{9}
\int_{\R_{>0}^n}\int_0^\infty \left|\exp\(-h_1(z,t)s+h_2(t)\)\right|ds dt.
\end{equation}
For $z\in -\exp(-\I\phi)H$, we have $\Re(\I ze^{\I\phi})>0$.
This inequality and the assumption
$\Re(-\I e^{\I(\phi+\theta r)})\geq 0$ and
$\Re(-\I e^{\I(\phi+\theta'r)})\geq 0$
imply $\Re h_1(z,t)>0$.
Hence, the integral \eqref{9} equals to
\begin{equation}\label{10}
\int_{\R_{>0}^n}\left|\exp\(h_2(t)\)\right|
\int_0^\infty \exp\(-s\Re h_1(z,t)\) ds dt
=
\int_{\R_{>0}^n}\frac{\left|\exp\(h_2(t)\)\right|}{\Re h_1(z,t)}dt.
\end{equation}
By the assumptions, the integral \eqref{10} is bounded from above by
\begin{align*}
\frac{1}{\Re(\I ze^{\I\phi})}
\int_{\R_{>0}^n}\left|\exp\(h_2(t)\)\right|dt
&=
\frac{1}{\Re(\I ze^{\I\phi})}
\left(\int_0^\infty e^{-\frac{1}{2}\cos(2\theta )t^2}dt\right)^k
\left(\int_0^\infty e^{-\frac{1}{2}\cos(2\theta')t^2}dt\right)^{n-k}\\
&=
\frac{1}{\Re(\I ze^{\I\phi})}
\left(\frac{\pi}{2\cos(\theta )}\right)^{k/2}
\left(\frac{\pi}{2\cos(\theta')}\right)^{(n-k)/2}
<\infty.
\end{align*}
Therefore, the integral \eqref{17} converges.
\end{proof}

\begin{lemma}
Suppose $\Im(e^{\I(\phi_k+\theta r)})\geq 0$ and $\Im(e^{\I(\phi_k+\theta'r)})\geq 0$
for $k=1,2$.
The equation 
$$\psi(z;\theta,\theta',\phi_1)=\psi(z;\theta,\theta',\phi_2)$$
holds for $z\in\(-\exp(-\I\phi_1)\)\cap\(-\exp(-\I\phi_2)\)$.
\end{lemma}
\begin{proof}
Put
\begin{align*}
\gamma_{1,R}(t) &:= te^{\I\phi_1} \,(0 \leq t \leq R), &
\gamma_{2,R}(t) &:= te^{\I\phi_2} \,(0 \leq t \leq R), &
\gamma_{3,R}(\phi) &:= Re^{\I\phi} \,(\phi_1 \leq \phi \leq \phi_2).
\end{align*}
By the Cauchy's integral formula, we have
\begin{equation}\label{12}
\int_{\gamma_{1,R}} \Phi dw
-\int_{\gamma_{2,R}} \Phi dw
+\int_{\gamma_{3,R}} \Phi dw
=0,
\end{equation}
where we put
$
\Phi =  e^{-\I w z}\left(\varphi_\theta(w)-\varphi_{\theta'}(w)\right)^n.
$
By some calculations, we have
$$
\left| \int_{\gamma_{3,R}} \Phi dw \right|
\leq M R\exp\(-R\)
$$
for sufficiently large $M>0$.
Taking the limit of \eqref{12} as $R$ approaches inifinity,
we have
$
\psi(z;\theta,\theta',\phi_1)-\psi(z;\theta,\theta',\phi_2)=0.
$
\end{proof}

\begin{theorem}\label{c8Jun2018}
Let $\varepsilon>0$ be a sufficiently small positive number
and $x$ be a continuous point of $f_X$.
When $r$ is odd, we have
$$
f_X(x)=
\begin{cases}
 \psi(x; \varepsilon,\pi            ,   -\varepsilon r)
-\psi(x;-\varepsilon,\pi            ,\pi+\varepsilon r)
& (x>0),\\
 \psi(x;           0,\pi-\varepsilon,    \varepsilon r)
-\psi(x;           0,\pi+\varepsilon,\pi-\varepsilon r)
& (x<0).
\end{cases}
$$
When $r$ is even, we have
$$
f_X(x)=
\begin{cases}
 \psi(x; \varepsilon,\pi+\varepsilon,   -\varepsilon r)
-\psi(x;-\varepsilon,\pi-\varepsilon,\pi+\varepsilon r)
& (x>0),\\
 \psi(x;           0,\pi            ,    \varepsilon r)
-\psi(x;           0,\pi            ,\pi-\varepsilon r)
& (x<0).
\end{cases}
$$
\end{theorem}
\begin{proof}
When $r$ is odd and $x>0$, we have
\begin{align*}
f_X(x)&=
\frac{1}{2\pi}
\lim_{y\rightarrow +0}\left(
  \psi(x-\I y; 0,\pi,0)
 -\psi(x+\I y; 0,\pi,\pi)
\right)\\
&=
\frac{1}{2\pi}
\lim_{y\rightarrow +0}\left(
  \psi(x-\I y; \varepsilon,\pi,0)
 -\psi(x+\I y; 0,\pi+\varepsilon,\pi)
\right)\\
&=
\frac{1}{2\pi}
\lim_{y\rightarrow +0}\left(
  \psi(x-\I y; \varepsilon,\pi,-\varepsilon r)
 -\psi(x+\I y; 0,\pi+\varepsilon,\pi+\varepsilon r)
\right)\\
&=
\frac{1}{2\pi}
\(
  \psi(x; \varepsilon,\pi,-\varepsilon r)
 -\psi(x; 0,\pi+\varepsilon,\pi+\varepsilon r)
\)
\end{align*}

When $r$ is odd and $x<0$, we have
\begin{align*}
f_X(x)&=
\frac{1}{2\pi}
\lim_{y\rightarrow +0}\left(
  \psi(x-\I y; 0,\pi,0)
 -\psi(x+\I y; 0,\pi,\pi)
\right)\\
&=
\frac{1}{2\pi}
\lim_{y\rightarrow +0}\left(
  \psi(x-\I y; -\varepsilon,\pi,0)
 -\psi(x+\I y; 0,\pi-\varepsilon,\pi)
\right)\\
&=
\frac{1}{2\pi}
\lim_{y\rightarrow +0}\left(
  \psi(x-\I y; -\varepsilon,\pi,\varepsilon r)
 -\psi(x+\I y; 0,\pi-\varepsilon,\pi-\varepsilon r)
\right)\\
&=
\frac{1}{2\pi}
\(
  \psi(x; -\varepsilon,\pi,\varepsilon r)
 -\psi(x; 0,\pi-\varepsilon,\pi-\varepsilon r)
\)
\end{align*}

When $r$ is even and $x>0$, we have
\begin{align*}
f_X(x)&=
\frac{1}{2\pi}
\lim_{y\rightarrow +0}\left(
  \psi(x-\I y; 0,\pi,0)
 -\psi(x+\I y; 0,\pi,\pi)
\right)\\
&=
\frac{1}{2\pi}
\lim_{y\rightarrow +0}\left(
  \psi(x-\I y; \varepsilon,\pi+\varepsilon,0)
 -\psi(x+\I y; 0,\pi+\varepsilon,\pi)
\right)\\
&=
\frac{1}{2\pi}
\lim_{y\rightarrow +0}\left(
  \psi(x-\I y; \varepsilon,\pi+\varepsilon,-\varepsilon r)
 -\psi(x+\I y; 0,\pi+\varepsilon,\pi+\varepsilon r)
\right)\\
&=
\frac{1}{2\pi}
\(
  \psi(x; \varepsilon,\pi+\varepsilon,-\varepsilon r)
 -\psi(x; 0,\pi+\varepsilon,\pi+\varepsilon r)
\)
\end{align*}

When $r$ is odd and $x<0$, we have
\begin{align*}
f_X(x)&=
\frac{1}{2\pi}
\lim_{y\rightarrow +0}\left(
  \psi(x-\I y; 0,\pi,0)
 -\psi(x+\I y; 0,\pi,\pi)
\right)\\
&=
\frac{1}{2\pi}
\lim_{y\rightarrow +0}\left(
  \psi(x-\I y; -\varepsilon,\pi,0)
 -\psi(x+\I y; 0,\pi-\varepsilon,\pi)
\right)\\
&=
\frac{1}{2\pi}
\lim_{y\rightarrow +0}\left(
  \psi(x-\I y; -\varepsilon,\pi,\varepsilon r)
 -\psi(x+\I y; 0,\pi-\varepsilon,\pi-\varepsilon r)
\right)\\
&=
\frac{1}{2\pi}
\(
  \psi(x; -\varepsilon,\pi,\varepsilon r)
 -\psi(x; 0,\pi-\varepsilon,\pi-\varepsilon r)
\)
\end{align*}
\end{proof}

\begin{remark}
The all of integral representations in Theorem \ref{c8Jun2018}
make sense as Lebesgue integral.
As we have shown in Lemma \ref{d8Jun2018},
when we write the integral in Theorem \ref{c8Jun2018}
into an integral on $\mathbf R$,
the integrand is a Lebesgue integrable function.
\end{remark}

\normalsize

\medbreak
{\it Acknowledgements}:
We are grateful to K. Yajima, T. Oshima, and
S.-J. Matsubara-Heo 
for valuable comments to Proposition \ref{a7Jun2018}.
This work was supported by MEXT/JSPS KAKENHI Grand Numbers JP
25220001, 
18J01507. 

\iftrue

\else
  \bibliographystyle{plain}
  \bibliography{reference}

\begin{thebibliography}{10}

\bibitem{castano-lopez}
A.~Casta{\~n}o-Mart\'inez and F.~L\'opez-Bl\'azquez.
\newblock Distribution of a sum of weighted noncentral chi-square variables.
\newblock {\em TEST}, 14(2):397--415, December 2005.

\bibitem{graf}
U.~Graf.
\newblock {\em Introduction to Hyperfunctions and Their Integral Transforms}.
\newblock {Birkh\"auser}, Basel, 2010.

\bibitem{dojo}
T.~Hibi.
\newblock {\em Gr{\"o}bner Bases: Statistics and Software System}.
\newblock Springer Japan, Tokyo, 2014.

\bibitem{imai}
I.~Imai.
\newblock {\em Applied hyperfunction theory}.
\newblock Kluwer Academic Publishers, Dordrecht, 1992.

\bibitem{kaneko1988}
A.~Kaneko.
\newblock {\em Introduction to Hyperfunctions}.
\newblock KTK Scientific Publishers, Tokyo, 1988.

\bibitem{kawai1970}
T.~Kawai.
\newblock On the theory of fourier hyperfunctions and its applications to
  partial differential equations with constant coefficients.
\newblock {\em Journal of the Faculty of Science, the University of Tokyo},
  17(3):467--517, December 1970.

\bibitem{takemura-koyama}
T.~Koyama and A.~Takemura.
\newblock Holonomic gradient method for distribution function of a weighted sum
  of noncentral chi-square random variables.
\newblock {\em Computational Statistics}, 31(4):1645--1659, December 2016.

\bibitem{marumo-oaku-takemura2015}
N.~Marumo, T.~Oaku, and A.~Takemura.
\newblock Properties of powers of functions satisfying second-order linear
  differential equations with applications to statistics.
\newblock {\em Japan Journal of Industrial and Applied Mathematics},
  32:553--572, July 2015.

\bibitem{NNNOSTT2011}
H.~Nakayama, K.~Nishiyama, M.~Noro, K.~Ohara, T.~Sei, N.~Takayama, and
  A.~Takemura.
\newblock Holonomic gradient descent and its application to the
  {F}isher-{B}ingham integral.
\newblock {\em Advances in Applied Mathematics}, 47:639--658, 2011.

\bibitem{ogata-hirayama}
H.~Ogata and H.~Hirayama.
\newblock Hyperfunction method for numerical integrations.
\newblock {\em Transactions of the Japan Society for Industrial and Applied
  Mathematics}, 26(1):33--43, 2016.
\newblock (in Japanese).

\bibitem{sato1958}
M.~Sato.
\newblock Theory of hyperfunctions.
\newblock {\em S\^{u}gaku}, 10:1--27, 1958.
\newblock in Japanese.

\bibitem{williams1991probability}
David Williams.
\newblock {\em Probability with Martingales}.
\newblock Cambridge University Press, Cambridge, 1991.

\end{thebibliography}
\fi

\end{document}